\documentclass[reqno,12pt]{amsart}

\usepackage{psfig}

\newcommand{\qDef}{{\mathcal Q}}
\makeatletter
\newcommand{\BFGS}{{\small BFGS}}
\newcommand{\LBFGS}{{\small L-BFGS}}
\newcommand{\mystrut}{\vrule height9.5pt depth1.5pt width0pt}
\newcommand{\MSS}{{\small MSS}}
\newcommand{\MATLAB}{{\small MATLAB}}
\newcommand{\mgap}{\;\;}
\newcommand{\bgap}{\;\;\;}
\newcommand{\words}[1]{\mgap\text{#1}\mgap}
\newcommand{\defined}{\mathop{\,{\scriptstyle\stackrel{\triangle}{=}}}\,}
\newcommand{\minimize}[1]{{\displaystyle\minim_{#1}}}
\newcommand{\minim}{\mathop{\operator@font{minimize}}}
\newcommand{\etal}{et al.}  
\newcommand{\subject}{\mathop{\operator@font{subject\ to}}} 
\newcommand{\DO}{\hskip2pt\textbf{do}\hskip2pt}

\newcommand{\END}{\textbf{end}}
\newcommand{\FOR}{\textbf{for}}

\newcommand{\epssigma}{\varepsilon}

\newcommand{\strt}{\rule[-.5ex]{0pt}{3ex}}
\newcommand{\hstrt}{\rule[-1ex]{0pt}{3.5ex}}
\newcommand{\CUTEr}{{\small CUTE}r}
\newcommand{\ENDDO}{\textbf{end do}}

\newcommand{\ELSE}{\hskip2pt\textbf{else}\hskip2pt}

\newcommand{\trace}{\mathop{\mathrm{trace}}}
\newcommand{\IF}{\textbf{if}\hskip2pt}
\newcommand{\WHILE}{\textbf{while}\hskip2pt}

\newcommand{\RETURN}{\mathbf{return}}

\newcommand{\tab}{\par\noindent\mystrut}
\newcommand{\tabb}{\tab\hskip1.5em}
\newcommand{\tabbb}{\tab\hskip3.5em}
\newcommand{\tabbbb}{\tab\hskip5.5em}
\newcommand{\card}{\mathop{\mathrm{card}}}
\newcommand{\Sscr}{{\mathcal S}}
\newcommand{\Pscr}{{\mathcal P}}
\newcommand{\st}{:}

\makeatother

\input{jen_macros.sty}

\begin{document}

\newtheorem{theorem}{Theorem}

\title[The MSS Method]{MSS: MATLAB
  Software for L-BFGS Trust-Region Subproblems for Large-Scale
  Optimization}

\author{Jennifer B. Erway}
\email{erwayjb@wfu.edu}
\address{Department of Mathematics, Wake Forest University, Winston-Salem, NC 27109}

\author{Roummel F.\ Marcia}
\email{rmarcia@ucmerced.edu}
\address{Appied Mathematics, University of California, Merced, Merced, CA 95343}

\date{\today}


\thanks{J.~B. Erway is supported in part by National Science Foundation grant
CMMI-1334042.
R.~F. Marcia is supported in part by National Science Foundation grant
CMMI-1333326.}

\begin{abstract}
  A MATLAB implementation of the Mor\'{e}-Sorensen sequential (MSS) method
  is presented.  The MSS method computes the minimizer of a quadratic
  function defined by a limited-memory BFGS matrix subject to a two-norm
  trust-region constraint.  This solver is an adaptation of the
  Mor\'{e}-Sorensen direct method into an L-BFGS setting for large-scale
  optimization.  The MSS method makes use of a recently proposed stable
  fast direct method for solving large shifted BFGS systems of
  equations~\cite{ErwaM12a,ErwaM12c} and is able to compute solutions to
  any user-defined accuracy.  This MATLAB implementation is a matrix-free
  iterative method for large-scale optimization.  Numerical experiments on
  the CUTEr~\cite{BonCGT95,GouOT03}) suggest that using the MSS method as a
  trust-region subproblem solver can require significantly fewer
  function and gradient evaluations needed by a trust-region method as
  compared with the Steihaug-Toint method.
\end{abstract}

\maketitle

\section{Introduction}
In this paper we describe a \MATLAB{} implementation for minimizing a
quadratic function defined by a limited-memory \BFGS{} (\LBFGS) matrix 
subject to a two-norm constraint, i.e., for a given $x_k$,
\begin{equation} \label{eqn-trustProblem}
  \minimize{p\in\Re^n}\mgap\qDef(p) \defined g^Tp + \frac{1}{2} p^TB p
   \bgap\subject \mgap \|p\|_2 \le \delta,
\end{equation}
where $g\defined\nabla f(x_k)$, $B$ is an \LBFGS{} approximation to
$\nabla^2 f(x_k)$, and $\delta$ is a given positive constant.
Approximately solving (\ref{eqn-trustProblem}) is of interest to the
optimization community, as it is the computational bottleneck of
trust-region methods for large-scale optimization.

Generally speaking, there is a trade-off in computational cost per
subproblem and the number of overall trust-region iterations (i.e.,
function and gradient evaluations): The more accurate the subproblem
solver, the fewer overall iterations required.  Solvers that reduce the
overall number of function and gradient evaluations are of particular
interest when function (or gradient) evaluations are time-consuming, e.g.,
simulation-based optimization.  Some solvers such as the ``dogleg'' method
(see~\cite{Pow70c,Pow70d}) or the ``double-dogleg'' strategy
(see~\cite{DenM79}) compute an \emph{approximate} solution to
(\ref{eqn-trustProblem}) by taking convex combinations of the steepest
descent direction and the Newton step.  Other solvers seek an approximate
solution to the trust-region subproblem using an iterative approach.  In
particular, the Steihaug-Toint method computes an approximate solution to
(\ref{eqn-trustProblem}) that is guaranteed to achieve at least half the
optimal reduction in the quadratic function when the model is
convex~\cite{Yua00,GouLRT99}, but does not specifically seek to solve the
minimization problem to high accuracy when the solution lies on the
boundary.  This paper presents an algorithm to solve
(\ref{eqn-trustProblem}) to any user-defined accuracy.

Methods to solve the trust-region subproblem to high accuracy are often
based on optimality conditions
given in the following theorem (see, e.g., Gay~\cite{Gay81},
Sorensen~\cite{Sor82}, Mor\'e and Sorensen \cite{MorS83} or Conn, Gould and
Toint \cite{ConGT00a}):
\begin{theorem}\label{thrm-optimality}
  Let $\delta$ be a positive constant.  A vector $p^*$ is a global
  solution of the trust-region subproblem (\ref{eqn-trustProblem}) if and only
  if $\|p^*\|_2\leq \delta$ and there exists a unique $\sigma^*\ge 0$ such
  that $B+\sigma^* I$ is positive semidefinite and
\begin{equation}\label{eqn-optimality}
(B+\sigma^* I)p^*=-g \mgap \words{and} \mgap \sigma^*(\delta-\|p^*\|_2)=0.
\end{equation}
Moreover, if $B+\sigma^* I$ is positive definite, then the global
minimizer is unique.
\end{theorem}

The Mor\'{e}-Sorensen algorithm~\cite{MorS83} seeks $(p^*,\sigma^*)$ that
satisfy the optimality conditions (\ref{eqn-optimality}) by trading off
between updating $p$ and $\sigma$.  That is, each iteration, the method
updates $p$ (fixing $\sigma$) by solving the linear system $(B+\sigma
I)p=-g$ using the Cholesky factorization of the $B+\sigma I$; then,
$\sigma$ is updated using a safeguarded Newton method to find a root of
\begin{equation}\label{eqn-pole}
\phi(\sigma)\defined
\frac{1}{\|p(\sigma)\|_2} - \frac{1}{\delta}.
\end{equation}
The Mor\'{e}-Sorensen direct method is arguably the best direct method for
solving the trust-region subproblem; in fact, the accuracy of each solve
can be specified by the user.  While this method is practical for
smaller-sized problems, in large-scale optimization it is too
computationally expensive to compute and store Cholesky factorizations for
unstructured Hessians.

Several researchers have proposed adaptations of the Mor\'{e}-Sorensen
direct method into the limited-memory \BFGS{} setting.  Burke
\etal~\cite{BurWX96} derive a method via the Sherman-Morrison-Woodbury
formula that uses two $M\times M$ Cholesky factorizations, where $M$ is the
number of limited-memory updates.
While this technique is able to exploit properties of \LBFGS{} updates,
there are potential instability issues related to their proposed use of the
Sherman-Morrison-Woodbury that are not addressed.  Lu and
Monteiro~\cite{LuMo07} also explore a Mor\'{e}-Sorensen method
implementation when $B$ has special structure; namely, $B = D+VEV^T$, where
$D$ and $E$ are positive diagonal matrices, and $V$ has a small number of
columns.  Their approach uses the Sherman-Morrison-Woodbury formula to
replace solves with $(B+\sigma I)$ with solves with an $M\times M$ system
composed of a diagonal plus a low rank matrix, and thus, avoid computing
Choleksy factorizations.  Like with~\cite{BurWX96}, there are potential
stability issues that are not addressed regarding inverting the $M\times M$
matrix.  

Finally, Apostolopoulou \etal{}~\cite{ApoSP08,ApoSP11} derive a
closed-form expression for $(B+\sigma I )^{-1}$ to solve the first equation
in (\ref{eqn-optimality}).  The authors are able to explicitly compute the
eigenvalues of $B$, provided $M=1$~\cite{ApoSP08,ApoSP11} or
$M=2$~\cite{ApoSP11}.  While their formula avoids potential instabilities
associated the Sherman-Morrison-Woodbury formula, their formula is
restricted to the case when the number of updates is at most two.

\subsection{Overview of the proposed methods}
In this paper, we describe a new adaptation of the Mor\'{e}-Sorensen solver
into a large-scale \LBFGS{} setting.  The proposed method, called the
\emph{Mor\'{e}-Sorensen sequential (\MSS) method}, is able to exploit the
structure of \BFGS{} matrices to solve the shifted \LBFGS{} system in
(\ref{eqn-optimality}) using a fast direct recursion method that the
authors originally proposed in~\cite{ErwaM12a}.  (This recursion was later
proven to be stable in~\cite{ErwaM12c}.)  The \MSS{} method is able to
solve (\ref{eqn-trustProblem}) to any prescribed accuracy.  A practical
trust-region implementation is given in the numerical results section;
less stringent implementations will be addressed in future work.


The paper is organized in five sections.  In Section~\ref{sec-back} we
review \LBFGS{} quasi-Newton matrices and introduce notation that will be
used for the duration of this paper.  Section~\ref{sec-numerical} includes
numerical results comparing the Mor\'{e}-Sorensen method and the \MSS{}
method.  Finally, Section~\ref{sec-conclusions} includes some concluding
remarks and observations.

\subsection{Notation and Glossary}\label{sec-notation}

Unless explicitly indicated, $\|\cdot\|$ denotes the vector two-norm or
its subordinate matrix norm.  
In this paper, all methods use \LBFGS{} updates and we assume they are
selected to ensure the quasi-Newton matrices remain sufficiently positive
definite (see, e.g., \cite{BurWX96}).

\section{Background}\label{sec-back}
In this section, we begin with an overview of the \LBFGS{} quasi-Newton
matrices described by Nocedal~\cite{Noc80}, defining notation that will be
used throughout the paper. 

The \LBFGS{} quasi-Newton method generates a
sequence of positive-definite matrices $\{B_j\}$ from a sequence of vectors
$\{y_j\}$ and $\{s_j\}$ defined as
$$y_j=\nabla f(x_{j+1})-\nabla f(x_j)\quad \words{and} \quad s_j=x_{j+1}-x_j,$$ 
where $j=0,\dots m-1$ where $m\leq M$, and $M$ is the maximum number of
allowed stored pairs $(y_j,s_j)$.  This method can be viewed as the \BFGS{}
quasi-Newton method where no more than the $M$ most recently computed
updates are stored and used to update an initial matrix $B_0$.  The \LBFGS{}
quasi-Newton approximation to the Hessian of $f$ is implicitly updated as
follows:
\begin{equation}
B_{m} = B_0 - \sum_{i=0}^{m-1} a_i {a_i}^T + \sum_{i=0}^{m-1} b_ib_i^T,
\label{eqn-bfgs}
\end{equation}
where 
\begin{equation}
a_i=\frac{B_is_i}{\sqrt{s_i^TB_is_i}}, 
\quad b_i = \frac{y_i}{\sqrt{y_i^T s_i}}, \quad 
B_0=\gamma_m^{-1} I, \label{eqn-bfgs-extra}
\end{equation}
and $\gamma_m>0$ is a constant.  In practice, $\gamma_m$ is often defined to
be $\gamma_m\defined s_{m-1}^Ty_{m-1}/\|y_{m-1}\|^2$ (see,
e.g.,~\cite{LiuN89} or~\cite{Noc80}).  In order to maintain that the sequence
$\{B_i\}$ is positive definite for $i=1,\ldots m$, each of the accepted
pairs must satisfy $y_i^Ts_i>0$ for $i=0,\ldots, m-1$.

%

One of the advantages of using an \LBFGS{} quasi-Newton is that there is an
efficient recursion relation to compute products with $B_m^{-1}$.  Given a
vector $z$, the following algorithm~\cite{Noc80,NocW06} terminates with
$r\defined B_m^{-1}z$: \newline \newline

\begin{Pseudocode}{Algorithm 1: Two-loop recursion to compute $r=B_m^{-1}z$}
  \tab $q\leftarrow z$;
  \tab \FOR\ $k=m-1,\dots,0$
  \tabb $\rho_k \leftarrow 1/(y_k^Ts_k)$;
  \tabb $\alpha_k\leftarrow \rho_ks_k^Tq$;
  \tabb $q \leftarrow q-\alpha_ky_k$;
  \tab \END\ 
  \tab $r\leftarrow B_0^{-1} q$;
  \tab \FOR\ $k=0,\ldots, m-1$
  \tabb  $\beta \leftarrow \rho_ky_k^Tr$;
  \tabb  $r \leftarrow r+(\alpha_k-\beta)s_k$:
  \tab \END\
\end{Pseudocode}

The two-term recursion formula requires at most $\mathcal{O}(Mn)$
multiplications and additions.  To compute products with the \LBFGS{}
quasi-Newton matrix, one may use the so-called ``unrolling'' formula, which
requires $\mathcal{O}(M^2n)$ multiplications, or one may use a compact
matrix representation of the \LBFGS{} that can be used to compute products
with the \LBFGS{} quasi-Newton matrix, which requires $\mathcal{O}(Mn)$
multiplications (see, e.g., ~\cite{NocW06}).  Further details on \LBFGS{}
updates can found in~\cite{NocW06}; further background on the BFGS{}
updates can be found in~\cite{DenS96}.

Without loss of generality, for the duration of the paper we assume that
$B$ is a symmetric positive-definite quasi-Newton matrix formed using $m$
($m\le M$) \LBFGS{} updates.

\section{The Mor\'{e}-Sorensen Sequential (\MSS) Method}\label{sec-mss}
In this section, we present the \MSS{} method to solve the constrained
optimization problem (\ref{eqn-trustProblem}).
We begin by considering the Mor\'{e}-Sorensen direct
method proposed in~\cite{MorS83}.

The Mor\'{e}-Sorensen direct method seeks a pair $(p,\sigma)$ that satisfy
the optimality conditions (\ref{eqn-optimality}) by alternating between
updating $p$ and $\sigma$.  In the case that the $B$ 
is positive definite (as in \LBFGS{} matrices), the method
simplifies to Algorithm 2~\cite{MorS83}.

\begin{Pseudocode}{Algorithm 2: Mor\'{e}-Sorensen Method}
\tab $\sigma\gets 0; \mgap p\gets -B^{-1}g$;
\tab \IF\ $\|p\|\leq \delta$ 
\tabb  $\RETURN$;
\tab \ELSE\     
\tabb   \WHILE\ $not$ converged \DO\
\tabbb     Factor $B+\sigma I= R^T R$;
\tabbb    Solve $R^TRp = -g$;
\tabbb    Solve $R^Tq = p$;
\tabbb    $\sigma \gets \sigma + \frac{\|p\|^2}{\|q\|^2}\frac{\|p\|-\delta}{\delta}$;
\tabb   \ENDDO\
\tab \END
\end{Pseudocode}

The update to $\sigma$ in Algorithm 2 can be shown to be
Newton's method applied (\ref{eqn-pole}).  (Since $B$ is positive definite,
a safeguarded Newton's method is unnecessary.)  Convergence is predicated on
solving the optimality conditions (\ref{eqn-optimality}) to a prescribed
accuracy.  The only difficulty in implementing the Mor\'{e}-Sorensen method
in a large-scale setting is the shifted solve $(B+\sigma I)p=-g$.

One method to directly solve systems of the form $(B+\sigma I)x=y$ is to
view the system matrix as the sum of $\sigma I$ and rank-one \LBFGS{}
updates to an initial diagonal matrix $B_0$.  It is important to distinguish
between \emph{applying} rank-one \LBFGS{} updates to $B_0+\sigma I$ and
\emph{viewing} the system matrix as the sum of rank-one updates to
$B_0+\sigma I$.  To compute $(B+\sigma I)^{-1}y$, one cannot simply
substitute $B_0+\sigma I$ in for $B_0$ in Algorithm 1.
(For a discussion on this, see~\cite{ErwaM12a}).  In ~\cite{ErwaM12a},
Erway and Marcia present a stable fast direct method for solving \LBFGS{}
systems that are shifted by a constant diagonal matrix (stability is shown
in~\cite{ErwaM12c}).  Specifically, it is shown that products with
$(B+\sigma I)^{-1}$ can be computed provided $\gamma\sigma$ is bounded away
from zero, where $B_0\defined \gamma^{-1} I$.  The following theorem
is found in~\cite{ErwaM12a}:

\begin{theorem} Let $\gamma>0$ and $\sigma\ge 0$.  Suppose
  $G=(\gamma^{-1}+\sigma )I$, and let $H=\sum_{i=0}^{2m-1} E_i$, where 
$$E_0=-a_0a_0^T, \mgap
E_1=b_0b_0^T, \mgap \ldots, \mgap E_{2m-2}=-a_{m-1}a_{m-1}^T, \mgap
E_{2m-1}=b_{m-1}b_{m-1}^T,$$ with $\{a_i\}$ and $\{b_i\}$ are defined as
in (\ref{eqn-bfgs-extra}).  Further, define $C_{m+1}=G+E_0+\cdots + E_m$.
If there exists some $\epsilon>0$ such that $\gamma\sigma>\epsilon$, then
$B+\sigma I=G+H$, and $(B+\sigma I)^{-1}$ is given by
$$(B+\sigma I)^{-1}=C_{2m-1}^{-1}-v_{2m-1}C_{2m-1}^{-1}E_{2m-1}C_{2m-1}^{-1}$$
where
$$C_{i+1}^{-1}=C_i^{-1}-v_iC_i^{-1}E_iC_{i}^{-1}, \mgap i=0,\ldots, 2m-1,$$
and
  $$v_i = \frac{1}{1+\trace \left(C_i^{-1}E_i\right) }.$$
\end{theorem}
\begin{proof}
See~\cite{ErwaM12a}.
\end{proof}

This theorem is the basis for the following algorithm derived
in~\cite{ErwaM12a} to compute products with $(B+\sigma I)^{-1}$.
The algorithm was shown to be stable in~\cite{ErwaM12c}.

\begin{Pseudocode}{Algorithm 3: Recursion to compute $x=(B+\sigma I)^{-1}y$}
\tab $x \leftarrow (\gamma^{-1} + \sigma)^{-1}y$;
\tab \FOR\ $k = 0, \dots, 2m-1$
\tabb  \IF\ $k$ even 
\tabbb     $c \leftarrow a_{k/2}$;
\tabb   \ELSE\ 
\tabbb      $c \leftarrow b_{(k-1)/2}$;
\tabb   \END\
\tabb $r_k \leftarrow (\gamma^{-1}+\sigma)^{-1}c$;
\tabb  \FOR\ $i = 0, \dots, k-1$
\tabbb    $r_k \leftarrow r_k + (-1)^iv_i (r_i^Tc) r_i$;
\tabb  \END\
\tabb $v_k \leftarrow 1/(1 + (-1)^{k+1}r_k^Tc)$;
\tabb $x \leftarrow x + (-1)^{k}v_k (r_k^Ty) r_k$;
\tab \END\
\end{Pseudocode}

It is important to note that $\{a_i\}$ and $\{b_i\}$ need to be
precomputed.  Algorithm 3 requires at most
$\mathcal{O}(M^2n)$.  Operations with $C_0=B_0+\sigma I$ and $C_1$ can be
easily computed with minimal extra expense since $C_0^{-1}$ is a diagonal
matrix.  It is generally known that $M$ may be kept small (for example,
Byrd et al.~\cite{ByrNS94} suggest $M\in [3,7]$).  When $M^2\ll n$, the
extra storage requirements and computations are affordable.

 \subsection{Handling small $\sigma$}

$\epssigma$
 In this section we discuss the case of solving $(B+\sigma I)p=-g$ for
 small $\sigma$.  For stability, it is important to maintain
 $\gamma\sigma>\epssigma$ for a small positive $\epssigma$.
 Thus, in order to use Algorithm 3, we require that both
 $\gamma>\sqrt{\epssigma}$ and $\sigma>\sqrt{\epssigma}$.  The
 first requirement is easily met by thresholding $\gamma$, i.e.,
 $\gamma\gets\max\{\sqrt{\epssigma}, \gamma \}$.  (Recall that
 $y_i^Ts_i>0$ for each $i=0,\ldots m-1$, and thus, $\gamma>0$.)  When
 $\sigma\le\sqrt{\epssigma}$, we set $\sigma=0$ and use the two-loop
 recursion (Algorithm 1) to solve an unshifted \LBFGS{}
 system.

\subsection{The algorithm}
The MSS method adapts the Mor\'{e}-Sorensen method into an \LBFGS{} setting
by solving $(B+\sigma I)p=-g$ using a recursion method instead of a
Cholesky factorization.  When $\sigma$ is sufficiently large, the recently
proposed recursion algorithm (Algorithm 3) is used to
update $s$; otherwise, $\sigma\approx 0$ and the two-loop recursion
(Algorithm 1) is used.  Also, note that the
Mor\'{e}-Sorensen method updates $\sigma$ using Newton's method applied to
(\ref{eqn-pole}), i.e., 
\begin{equation}\label{eqn-update}
\sigma\gets\sigma-\phi(p)/\phi'(p).
\end{equation} 
In Algorithm 2 this update is written in terms of the Cholesky
factors.  In the \MSS{} algorithm, Cholesky factors are unavailable, and
thus, the update to $\sigma$ is performed by explicitly computing the
Newton step (\ref{eqn-update}).  For details on this update
see~\cite{ConGT00a}.

The \MSS{} method is summarized in Algorithm 4:

\begin{Pseudocode}{Algorithm 4: Mor\'{e}-Sorensen Sequential (\MSS) Method}\\
{\bf Input:} $\gamma>\sqrt{\epssigma}$ where $0<\epssigma \ll 1$\\
{\bf Output:} $(p,\sigma)$ \\
$\sigma\gets 0; \mgap p\gets -B^{-1}g$;
\tab \IF\ $\|p\|\leq \delta$
\tabb $\RETURN$;
\tab \ELSE\
\tabb  \WHILE\ $not$ converged \DO\
\tabbb    $\phi(p) \gets 1/\|p\| - 1/\delta;$
\tabbb    \IF\ $\sigma>\sqrt{\epssigma}$
\tabbbb   Compute $\hat{p}$ such that $(B+\sigma)\hat{p}=-p$ using Algorithm 3;
\tabbb     \ELSE\
\tabbbb        Compute $\hat{p}$ such that $B\hat{p}=-p$ using Algorithm 1;
\tabbbb        $\sigma\gets 0$;
\tabbb     \END\
\tabbb  $\phi'(p)\gets -(p^T\hat{p})/ \|p\|^3$;
\tabbb  $\sigma\gets \sigma - \phi(p) /\phi'(p)$;
\tabbb  \IF\ $\sigma>\sqrt{\epssigma}$
\tabbbb     Compute $p$ such that $(B+\sigma)p=-g$ using Algorithm 3;
\tabbb  \ELSE\
\tabbbb    Compute $p$ such that $Bp=-g$ using Algorithm 1;
\tabbbb    $\sigma\gets 0$;  
\tabbb  \END\
\tabb \END\
\tab \END\
\end{Pseudocode}

\subsection{Stopping criteria}

If the initial solution $p = -B^{-1}g$ is within the trust-region, i.e., $\|
p \| \le \delta$, the \MSS{} method terminates.  Otherwise, the \MSS{}
method iterates until the solution to \eqref{eqn-trustProblem} is
sufficiently close to the constraint boundary, i.e., $| \| p \| - \delta |
\le \tau \delta$, for a fixed $\tau \ll 1$.


 \section{Numerical Results}\label{sec-numerical}
 We test the efficiency of the MSS method by solving large problems from
 the \CUTEr{} test collection (see 
 \cite{BonCGT95,GouOT03}).  
 The test set was constructed using the \CUTEr{} interactive
 \texttt{select} tool, which allows the identification of groups of
 problems with certain characteristics.  In our case, the \texttt{select}
 tool was used to identify the twice-continuously differentiable
 unconstrained problems for which the number of variables can be varied.
 This process selected 67 problems:
\texttt{arwhead}, 
\texttt{bdqrtic}, 
\texttt{broydn7d}, 
\texttt{brybnd}, 
\texttt{chainwoo}, 
\texttt{cosine}, 
\texttt{cragglvy}, 
\texttt{curly10}, 
\texttt{curly20}, 
\texttt{curly30}, 
\texttt{dixmaana}, 
\texttt{dixmaanb}, 
\texttt{dixmaanc}, 
\texttt{dixmaand}, 
\texttt{dixmaane}, 
\texttt{dixmaanf}, 
\texttt{dixmaang}, 
\texttt{dixmaanh}, 
\texttt{dixmaani}, 
\texttt{dixmaanj}, 
\texttt{dixmaank}, 
\texttt{dixmaanl}, 
\texttt{dixon3dq}, 
\texttt{dqdrtic}, 
\texttt{dqrtic}, 
\texttt{edensch}, 
\texttt{eg2}, 
\texttt{engval1}, 
\texttt{extrosnb}, 
\texttt{fletchcr}, 
\texttt{fletcbv2}, 
\texttt{fminsrf2}, 
\texttt{fminsurf}, 
\texttt{freuroth}, 
\texttt{genhumps}, 
\texttt{genrose}, 
\texttt{liarwhd}, 
\texttt{morebv}, 
\texttt{ncb20}, 
\texttt{ncb20b}, 
\texttt{noncvxu2}, 
\texttt{noncvxun}, 
\texttt{nondia}, 
\texttt{nondquar}, 
\texttt{penalty1}, 
\texttt{penalty2}, 
\texttt{powellsg}, 
\texttt{power}, 
\texttt{quartc}, 
\texttt{sbrybnd}, 
\texttt{schmvett}, 
\texttt{scosine}, 
\texttt{scurly10}, 
\texttt{scurly20}, 
\texttt{scurly30}, 
\texttt{sinquad}, 
\texttt{sparsine}, 
\texttt{sparsqur}, 
\texttt{spmsrtls}, 
\texttt{srosenbr}, 
\texttt{testquad}, 
\texttt{tointgss}, 
\texttt{tquartic}, 
\texttt{tridia}, 
\texttt{vardim}, 
\texttt{vareigvl} and 
\texttt{woods}.  
The dimensions were selected so that $n\ge 1000$, with a default of
$n=1000$ unless otherwise recommended in the \CUTEr{} documentation.

The \MATLAB{} implementation of the \MSS{} method was tested against the
Steihaug-Toint method.
A basic trust-region algorithm (Algorithm 5) was
implemented based on (\cite[Algorithm 6.1]{GouLRT99}) with some extra
precautions based on (\cite{DenM79})
to prevent the trust-region radius from becoming too large.

\begin{Pseudocode}{Algorithm 5: Basic Trust-Region Algorithm}\\
{\bf Input:} $x_0$; $M>0$ (an integer); \,\, $\hat{\delta} \gg 0$; \,\,  $\gamma_1>1$; \,\, $\gamma_2\in (0,1)$; \,\,
$\delta_0>0$ \,\, and \,\, $\eta_1,\eta_2<1$,
such that $0<\eta_1<\eta_2<1$\\
{\bf Output:}  $x$ \\
$\delta\gets \delta_0$; $x\gets x_0$;\\
$B \gets I$; (implicit)\\
\WHILE\ $not$ converged \DO\
\tabb  Compute $p$ to (approximately) solve (\ref{eqn-trustProblem}); 
\tabb  Evaluate $f(x+p)$ and $g(x+p)$;
\tabb   $\rho\gets (f(x)-f(x+p))/(-g(x)^Tp-\frac{1}{2}p^TBp)$;
\tabb   \IF\ $\rho\ge\eta_1$ 
\tabbb    $x\gets x+p;$ \,\, $f(x)\gets f(x+p)$ \,\, $g(x)\gets g(x+p)$; 
\tabbb  \IF\ $\rho\ge\eta_2$ 
\tabbbb      $\delta\gets\min\{\gamma_1\|p\|,\hat{\delta}\}$;
\tabbb \ELSE\  
\tabbbb     $\delta\gets \|p\|$; 
\tabbb  \END\
\tabb   \ELSE\
\tabbb      (reject the update) 
\tabbb      $\delta\gets\gamma_2\delta$;
\tabb \END\
\tabb    Possibly update the pairs $\{(s_i,y_i)\}$, $i=1\ldots M$, to implicitly update $B$;\\
\END\
\end{Pseudocode}
 
The following termination criteria was used for the basic trust-region algorithm: 
\begin{equation}\label{eqn-bstr-termination}
\|g(x)\|<\max\{\tau\|f(x_0)\|,\tau\|g(x_0)\|, 1\times 10^{-5}\},\end{equation}
where $\tau \defined 1\times 10^{-6}$.  The trust-region
algorithm terminated unsuccessful whenever the trust-region radius
becomes too small or the number of function
evaluations exceeded $\max\{1000, n\}$, where $n$ is the problem size.

For these numerical experiments, the \LBFGS{} pairs were updated whenever the 
updates satisfied the following inequalities:
\begin{equation}\label{eqn-updateSY}
s_+^Ty_+<1/\sqrt{\epsilon} \quad \text{and} \quad
s_+^Ty_+>\sqrt{\epsilon},
\end{equation}
where $s_+ \defined p$, 
$y_+ \defined g(x+p)-g(x)$ 
and $\epsilon$ is machine precision in \MATLAB{} (i.e., \texttt{eps}).
The update conditions (\ref{eqn-updateSY}) were tested each iteration,
regardless of whether the update to $x$ was accepted.

Finally, the following constants were used for Algorithm 5:
$M\defined 5$, $\gamma_1\defined 2.0$, $\gamma_2\defined 0.5$, $\delta_0
\defined 1$, $\eta_1\defined 0.01$, $\eta_2\defined 0.95$, and
$\hat{\delta}\defined 1/(100\times \epsilon)$ where
$\epsilon$ is machine precision in \MATLAB{} (i.e., \texttt{eps}).  The initial
guess $x_0$ was the default starting point associated with each \CUTEr{}
problem.

\subsection{Termination criteria for the subproblem solvers}
The Steihaug-Toint truncated conjugate-gradient method as found
in~\cite{ConGT00a} was implemented in \MATLAB.  The maximum number of
iterations per subproblem was limited to the minimum of either $n$, the
size of each \CUTEr{} problem, or 100.  The subproblem solver was said to
converge at the $i$th iteration when the residual of the linear solve $r_i$
satisfied
$$\|r_i\|\leq\|g(x)\|\min\{0.1,\|g(x)\|^{0.1}\}$$ (see, e.g., ~\cite{ConGT00a}).

The \MSS{} method terminates whenever either of the following conditions hold
~\cite{MorS83}:
$$\|B^{-1}g\|\leq \delta \quad \text{or} \quad |\|p\|-\delta|\leq \tau_{ms}\delta,$$
where $\tau_{ms}$ is a small non-negative constant.  In the numerical
experiments, $\tau_{ms}\defined\sqrt{\epsilon}$, where $\epsilon$ denotes
\MATLAB{} machine precision.  Furthermore the value for $\epsilon$ in
Algorithm 4 was also taken to be \MATLAB{} machine precision.  The maximum
number of iterations per subproblem was limited to the minimum of $n$, the size of each
\CUTEr{} problem, or 100.

\subsection{Results}
Of the 67 \CUTEr{} problems the basic trust-region method combined with the
two subproblems solvers failed to solve the following problems:
\texttt{curly10}, 
\texttt{curly20}, 
\texttt{curly30}, 
\texttt{dixon3dq}, 
\texttt{fletchcr}, 
\texttt{genhumps}, 
\texttt{genrose}, 
\texttt{sbrybnd}, 
\texttt{scosine}, 
\texttt{scurly10}, 
\texttt{scurly20}, 
and
\texttt{scurly30}. 
The problems \texttt{fletcbv2} and \texttt{penalty2} satisfied the
convergence criteria (\ref{eqn-bstr-termination}) at the initial point
$x_0$ and thus, were removed from the test set.  Tables I and II
contain the results of the remaining 53 problems. 

\begin{table}[h]\label{table-cuter-1}
\caption{MSS method and Steihaug-Toint method on CUTEr problems A--E.} 
\centering
\begin{tabular}{|lc|cccccc|}
  \hline
  \multicolumn{2}{|c|}
  {Problem}  & \multicolumn{3}{c}    {MSS Method} &
\multicolumn{3}{c|} {Steihaug-Toint Method}\\
  name & $n$  &  Time & FEs & Inner Itns & Time & FEs & Inner Itns \\
  \hline

 \texttt{ARWHEAD   } &  5000 &  \texttt{1.11e-01} & \texttt{15} & \texttt{29} &
 \texttt{1.93e-02} & \texttt{15} & \texttt{24} \\  

 \texttt{BDQRTIC   } &  5000 &  \texttt{3.12e-01} & \texttt{40} & \texttt{69} &
 \texttt{1.06e-01} & \texttt{40} & \texttt{122} \\ 

 \texttt{BROYDN7D  } &  5000 &  \texttt{5.98e+00} & \texttt{1566} & \texttt{1014} &
 \texttt{7.67e+00} & \texttt{1618} & \texttt{8859} \\  

 \texttt{BRYBND    } &  5000 &  \texttt{2.56e-01} & \texttt{59} & \texttt{52} &
 \texttt{5.56e-02} & \texttt{24} & \texttt{66} \\ 

 \texttt{CHAINWOO  } &  4000 &  \texttt{4.18e-01} & \texttt{54} & \texttt{99} &
 \texttt{1.33e-01} & \texttt{46} & \texttt{167} \\ 

 \texttt{COSINE    } &  10000 &  \texttt{8.45e-03} & \texttt{14} & \texttt{14} &
 \texttt{1.02e-02} & \texttt{14} & \texttt{13} \\ 

 \texttt{CRAGGLVY  } &  5000 &  \texttt{2.78e-01} & \texttt{36} & \texttt{64} &
 \texttt{1.04e-01} & \texttt{39} & \texttt{122} \\ 

 \texttt{DIXMAANA  } &  3000 &  \texttt{2.32e-02} & \texttt{13} & \texttt{14} &
 \texttt{1.12e-02} & \texttt{13} & \texttt{17} \\ 

 \texttt{DIXMAANB  } &  3000 &  \texttt{2.25e-02} & \texttt{13} & \texttt{14} &
 \texttt{7.47e-03} & \texttt{13} & \texttt{12} \\ 

 \texttt{DIXMAANC  } &  3000 &  \texttt{2.32e-02} & \texttt{14} & \texttt{14} &
 \texttt{8.02e-03} & \texttt{14} & \texttt{13} \\ 

 \texttt{DIXMAAND  } &  3000 &  \texttt{2.54e-02} & \texttt{15} & \texttt{14} &
 \texttt{8.61e-03} & \texttt{15} & \texttt{14} \\ 

 \texttt{DIXMAANE  } &  3000 &  \texttt{1.76e-01} & \texttt{54} & \texttt{46} &
 \texttt{1.43e-01} & \texttt{52} & \texttt{213} \\ 

 \texttt{DIXMAANF  } &  3000 &  \texttt{4.61e-02} & \texttt{24} & \texttt{18} &
 \texttt{3.66e-02} & \texttt{24} & \texttt{55} \\ 

 \texttt{DIXMAANG  } &  3000 &  \texttt{2.97e-02} & \texttt{21} & \texttt{14} &
 \texttt{2.12e-02} & \texttt{21} & \texttt{33} \\ 

 \texttt{DIXMAANH  } &  3000 &  \texttt{2.96e-02} & \texttt{20} & \texttt{14} &
 \texttt{1.81e-02} & \texttt{20} & \texttt{27} \\ 

 \texttt{DIXMAANI  } &  3000 &  \texttt{4.99e-01} & \texttt{84} & \texttt{129} &
 \texttt{2.27e-01} & \texttt{75} & \texttt{340} \\ 

 \texttt{DIXMAANJ  } &  3000 &  \texttt{9.93e-02} & \texttt{28} & \texttt{30} &
 \texttt{4.82e-02} & \texttt{27} & \texttt{72} \\ 

 \texttt{DIXMAANK  } &  3000 &  \texttt{7.77e-02} & \texttt{24} & \texttt{27} &
 \texttt{3.70e-02} & \texttt{25} & \texttt{56} \\ 

 \texttt{DIXMAANL  } &  3000 &  \texttt{4.47e-02} & \texttt{22} & \texttt{18} &
 \texttt{1.97e-02} & \texttt{21} & \texttt{31} \\ 

 \texttt{DQDRTIC   } &  5000 &  \texttt{5.36e-02} & \texttt{11} & \texttt{20} &
 \texttt{7.12e-03} & \texttt{11} & \texttt{10} \\ 

 \texttt{DQRTIC    } &  5000 &  \texttt{6.54e-03} & \texttt{29} & \texttt{56} &
 \texttt{1.76e-02} & \texttt{29} & \texttt{28} \\ 

 \texttt{EDENSCH   } &  2000 &  \texttt{3.86e-02} & \texttt{22} & \texttt{23} &
 \texttt{1.74e-02} & \texttt{24} & \texttt{39} \\ 

 \texttt{EG2       } &  1000 &  \texttt{8.52e-04} & \texttt{5} & \texttt{2} &
 \texttt{9.72e-04} & \texttt{5} & \texttt{4} \\ 

 \texttt{ENGVAL1   } &  5000 &  \texttt{4.54e-02} & \texttt{17} & \texttt{17} &
 \texttt{1.63e-02} & \texttt{17} & \texttt{21} \\ 

 \texttt{EXTROSNB  } &  1000 &  \texttt{1.13e-01} & \texttt{39} & \texttt{67} &
 \texttt{5.78e-02} & \texttt{57} & \texttt{162} \\ 
\hline
\end{tabular}
\end{table}

\begin{table}[h]\label{table-cuter-2}
\caption{MSS method and Steihaug-Toint method on CUTEr problems F--W.}
\centering
\begin{tabular}{|lc|cccccc|}
  \hline
  \multicolumn{2}{|c|}
  {Problem}  & \multicolumn{3}{c}    {MSS Method} &
\multicolumn{3}{c|} {Steihaug-Toint Method}\\
  name & $n$  &  Time & FEs & Inner Itns & Time & FEs & Inner Itns \\
  \hline

 \texttt{FMINSRF2  } &  5625 &  \texttt{1.91e+00} & \texttt{569} & \texttt{268} &
 \texttt{3.26e+00} & \texttt{841} & \texttt{3513} \\ 

 \texttt{FMINSURF  } &  1024 &  \texttt{4.88e-01} & \texttt{241} & \texttt{210} &
 \texttt{7.68e-01} & \texttt{460} & \texttt{2182} \\ 

 \texttt{FREUROTH  } &  5000 &  \texttt{2.28e-01} & \texttt{36} & \texttt{61} &
 \texttt{6.68e-02} & \texttt{39} & \texttt{80} \\ 

 \texttt{LIARWHD   } &  5000 &  \texttt{1.95e-01} & \texttt{30} & \texttt{50} &
 \texttt{1.01e-01} & \texttt{69} & \texttt{116} \\ 

 \texttt{MOREBV    } &  5000 &  \texttt{1.70e-01} & \texttt{261} & \texttt{4} &
 \texttt{4.45e-01} & \texttt{262} & \texttt{703} \\ 

 \texttt{NCB20     } &  1010 &  \texttt{2.64e+00} & \texttt{772} & \texttt{1251} &
 \texttt{*} & \texttt{*} & \texttt{*} \\ 

 \texttt{NCB20B    } &  2000 &  \texttt{9.42e-02} & \texttt{49} & \texttt{38} &
 \texttt{1.34e-01} & \texttt{69} & \texttt{307} \\ 

 \texttt{NONCVXU2  } &  5000 &  \texttt{1.94e-01} & \texttt{19} & \texttt{53} &
 \texttt{1.45e-02} & \texttt{19} & \texttt{18} \\ 

 \texttt{NONCVXUN  } &  5000 &  \texttt{2.20e-01} & \texttt{19} & \texttt{55} &
 \texttt{1.60e-02} & \texttt{19} & \texttt{19} \\ 

 \texttt{NONDIA    } &  5000 &  \texttt{1.34e-03} & \texttt{4} & \texttt{2} &
 \texttt{2.09e-03} & \texttt{4} & \texttt{3} \\ 

 \texttt{NONDQUAR  } &  5000 &  \texttt{3.93e-01} & \texttt{47} & \texttt{85} &
 \texttt{1.41e-01} & \texttt{49} & \texttt{167} \\ 

 \texttt{PENALTY1  } &  1000 &  \texttt{2.89e-03} & \texttt{25} & \texttt{48} &
 \texttt{4.83e-03} & \texttt{25} & \texttt{24} \\ 

 \texttt{POWELLSG  } &  5000 &  \texttt{9.90e-02} & \texttt{29} & \texttt{28} &
 \texttt{5.33e-02} & \texttt{25} & \texttt{65} \\ 

 \texttt{POWER     } &  1000 &  \texttt{8.49e-03} & \texttt{37} & \texttt{56} &
 \texttt{1.15e-02} & \texttt{37} & \texttt{45} \\ 

\texttt{QUARTC    } &  5000 &  \texttt{6.38e-03} & \texttt{29} & \texttt{56} &
\texttt{9.98e-03} & \texttt{29} & \texttt{28} \\ 

 \texttt{SCHMVETT  } &  5000 &  \texttt{2.17e-01} & \texttt{46} & \texttt{48} &
 \texttt{6.68e-02} & \texttt{29} & \texttt{80} \\ 

 \texttt{SINQUAD   } &  5000 &  \texttt{2.73e-01} & \texttt{36} & \texttt{88} &
\, \texttt{3.23e-01}$^\dagger$ & \, \texttt{171}$^\dagger$ & \, \texttt{376}$^\dagger$ \\ 

 \texttt{SPARSINE  } &  5000 &  \texttt{6.72e-01} & \texttt{115} & \texttt{131} &
 \texttt{5.30e-01} & \texttt{126} & \texttt{620} \\ 

 \texttt{SPARSQUR  } &  10000 &  \texttt{4.77e-02} & \texttt{17} & \texttt{14} &
 \texttt{3.86e-02} & \texttt{18} & \texttt{30} \\ 

 \texttt{SPMSRTLS  } &  4999 &  \texttt{5.37e-01} & \texttt{114} & \texttt{107} &
 \texttt{5.05e-01} & \texttt{128} & \texttt{589} \\ 

 \texttt{SROSENBR  } &  5000 &  \texttt{1.51e-01} & \texttt{22} & \texttt{43} &
 \texttt{3.59e-02} & \texttt{27} & \texttt{41} \\ 

 \texttt{TESTQUAD  } &  5000 &  \texttt{2.77e-01} & \texttt{72} & \texttt{53} &
 \texttt{3.27e-01} & \texttt{78} & \texttt{387} \\ 

 \texttt{TOINTGSS  } &  5000 &  \texttt{5.06e-02} & \texttt{11} & \texttt{20} &
 \texttt{8.98e-03} & \texttt{12} & \texttt{12} \\ 

 \texttt{TQUARTIC  } &  5000 &  \texttt{2.34e+00} & \texttt{47} & \texttt{370} &
 \texttt{*} & \texttt{*} & \texttt{*} \\ 

 \texttt{TRIDIA    } &  5000 &  \texttt{2.19e+00} & \texttt{263} & \texttt{426} &
 \texttt{1.35e+00} & \texttt{291} & \texttt{1507} \\  

 \texttt{VARDIM    } &  200 &  \texttt{1.13e-03} & \texttt{12} & \texttt{22} &
 \texttt{2.01e-03} & \texttt{12} & \texttt{11} \\ 

 \texttt{VAREIGVL  } &  1000 &  \texttt{1.18e-01} & \texttt{31} & \texttt{70} &
 \texttt{2.65e-02} & \texttt{25} & \texttt{77} \\ 

 \texttt{WOODS     } &  4000 &  \texttt{1.25e-01} & \texttt{22} & \texttt{37} &
 \texttt{3.17e-02} & \texttt{22} & \texttt{39} \\ 

  \hline
\end{tabular}

\end{table}

Each table reports the total time spent in each subproblem solver for each
problem, the total number of function evaluations required using each
subproblem solver, and the total number of iterations performed by each
subproblem solver required to obtain a solution to the unconstrained
problem.  On the problems \texttt{ncb20} and \texttt{tquartic}, the
Steihaug-Toint method was unable to converge; this is denoted in the table
with asterisks.  On the problem \texttt{sinquad}, the basic-trust region
algorithm terminated early using the Steihaug-Toint method due to the
trust-region radius becoming too small ($3.78\times 10^{-14}$); this is
denoted in the table with a dagger.  It should be noted that it takes one
matrix-vector product per inner iteration for Steihaug's method; for
brevity's sake, only the number of inner iterations are reported.  It is
also worth noting that neither the \MSS{} method nor the Steihaug-Toint method
reached the limit on the maximum number of iterations per subproblem;
thus, each subproblem was solved before reaching the maximum number of
inner iterations.

Comparing the number of function evaluations and inner iterations between
Tables I and II, it appears that the more difficult problems occur in Table
II.  While the \MSS{} method often took slightly more time than the
Steihaug-Toint method, it often required fewer function evaluations on
these problems.  The results of Tables I and II are summarized in Table
III.  (The three problems on which the Steihaug-Toint method did not
converge were removed when generating Table III.)  The results suggest that
solving the trust-region subproblem more accurately using the \MSS{} method 
leads to fewer overall function evaluations for the overall trust-region
method.  This is consistent with other subproblem solvers that solve
trust-region subproblems to high accuracy~\cite{ErwGG09,ErwG09}.  Table III
also suggests that while the \MSS{} method took more time overall, it
solved the subproblems within a comparable time frame.

\begin{table}[ht]
\caption{Summary of results on CUTEr problems.}
\centering
\begin{tabular}{lcc}\hline
\hstrt                 & MSS method& Steihaug-Toint method\\ \hline
\strt  Problems solved     &    53   &     50   \\
\strt  Function evals      &  4359   &   4974  \\
\strt  Total time (subproblem) & 1.71e+01   &  1.68e+01  \\\hline
\end{tabular}

\end{table}

The results of Tables I and II are also
summarized using a performance profile (see Dolan and
Mor\'{e}~\cite{DolM02}).  Let $\card(\Sscr)$ denote the number of elements
in a finite set $\Sscr$.  Let $\Pscr$ denote the set of problems used for a
given numerical experiment.  For each method $s$ we define the function
$\pi_s \st [0,r_M]\mapsto\Re^+$ such that
$$
  \pi_s(\tau)
        = \dfrac{1}{\card(\Pscr)} \card(\{ p \in \Pscr \st \log_2(r_{p,s}) \le \tau \}),
$$
where $r_{p,s}$ denotes the ratio of the number of function evaluations
needed to solve problem $p$ with method $s$ and the least number of
function evaluations needed to solve problem $p$.  The number $r_M$ is the
maximum value of $\log_2(r_{p,s})$.  Figure~\ref{fig-fe2} depicts the
functions $\pi_s$ for each of the methods tested. All performance profiles
are plotted on a $\log_2$-scale.  Based on the performance profile, the
\MSS{} method appears more competitive in terms of function evaluations
than the Steihaug-Toint method for the \LBFGS{} application.

\begin{figure}[ht]
\begin{center}
\centering
\centerline{\psfig{figure=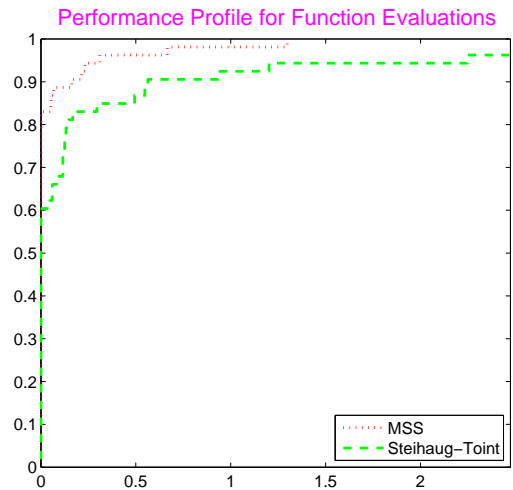,height=7cm}}
\parbox[t]{.85\textwidth}{
 \caption{\label{fig-fe2} \small Function evaluations for the \MSS{} and
Steihaug-Toint methods }}
\end{center}
\end{figure}

 \section{Concluding Remarks}\label{sec-conclusions}
 In this paper, we have presented a \MATLAB{} implementation of the \MSS{}
 algorithm that is able to solve problems of the form
 (\ref{eqn-trustProblem}).  This solver is stable and numerical results
 confirm that the method can compute solutions to any prescribed accuracy.
 Future research directions include practical
trust-region schemes that include less stringent requirements on
accepting updates to the \LBFGS{} matrix and improvements for cases when
L-BFGS pairs become linearly dependent, i.e., either the vectors $\{s_i\}$ or 
the vectors $\{y_i\}$, $i=1\ldots M$, become linearly dependent.

\bibliographystyle{abbrv}
\bibliography{my_refs}

\end{document}